\documentclass[12pt,a4paper, notitlepage]{article}
\usepackage[a4paper] {geometry}
\usepackage{amsmath,amsfonts,amssymb,mathrsfs,amsthm}
\usepackage{multirow}
\usepackage[english]{babel}
\usepackage[all]{xy}
\usepackage{tikz}
\numberwithin{equation}{section}
\numberwithin{figure}{section}
\newtheorem{thm}{Theorem}[section]
\newtheorem{prop}[thm]{Proposition}
\newtheorem{lem}[thm]{Lemma}
\newtheorem{coro}[thm]{Corollary}
\theoremstyle{definition}
\newtheorem{defx}[thm]{Definition}
\newtheorem{rem}[thm]{Remark}

\newcommand{\R}{I\!\!R}

\DeclareMathSizes{12}{13}{10}{10}

\title{Groupoid Characterization of Locally Convex Partial $^*$-Algebras}
\author{{\emph{N. O. Okeke \;  Email: nickokeke@dui.edu.ng, }}
\\ Physical and Mathematical Sciences, Dominican University, Ibadan \\
\emph{M. E. Egwe \;  murphy.egwe@ui.edu.ng}
\\ Department of Mathematics, University of Ibadan, Ibadan, Nigeria}

\date{\vspace{-5ex}}

\begin{document}
\maketitle

\begin{abstract}
Given a locally convex space $(\mathcal{A},\tau)$ with a Hausdorff locally convex topology $\tau$ such that the following maps are continuous; $u \mapsto u^*$ for all $u \in \mathcal{A}$, $x \mapsto x\cdot y$ and $x \mapsto z\cdot x$ for every left and right multipliers of $\mathcal{A}$. In this paper we re-characterized the locally convex partial $*$-algebra $(\mathcal{A}, \Gamma,\cdot,*,\tau)$ arising from these continuous maps in terms of convolution algebra of a Lie groupoid $\Gamma \rightrightarrows \mathcal{A}$. This is advantageous because the pathologies of the underlying spaces owing to their quantum mechanical nature are easily resolved in groupoid terms.
\end{abstract}
\let\thefootnote\relax
\let\thefootnote\relax

\section*{Introduction}
The motivation for this work is basically to use the more conducive groupoid structure to recharacterize the locally convex partial $^*$-algebra which was characterized by Ekhaguere in \cite{Ekhaguere2007}. The re-characterization is basically through the relation $\Gamma$ defined by the partial multiplication $(\cdot)$ on a locally convex linear space $\mathcal{A}$, which is used to define a groupoid $\Gamma \rightrightarrows \mathcal{A}$, where the diagonal $\Delta = \{(x,x) \subset \mathcal{A} \times \mathcal{A}\}$ is the space of objects isomorphic to $\mathcal{A}$. The local convexity of the linear space $\mathcal{A}$ offers various advantages for the analysis of the resulting groupoid convolution algebra. These advantages, according to \cite{SchmedingWockel2014}, include among other things, the following:
\begin{itemize}
\item
 It is naturally compatible with the underlying topological framework;
\item
 It makes for automatic continuity of the smooth maps and differentials on the framework which has important geometric application, especially in the definition of system of left Haar measures for the groupoid;
\item
 It aids the construction of a Lie group structure on the set of bisections which is related and derivable from the canonical smooth structure on the manifold of mappings.
\end{itemize}
These accord with the fact that in analysis, unbounded operators frequently occur when symmetries are introduced using Lie groups. In this case the algebra arising from a smooth net $\mathscr{K}$, which is shown to be isomorphic to a Lie group (we sometime also denoted the smooth algebra with $\mathscr{K}$), carries a global aspect of the partial symmetry of the structure which complements with that of the relation $\Gamma$. Thus, the unbounded operators usually appear as open (unbounded) infinitesimal generators which are mostly differential operators of infinite dimensional Lie pseudogroups. We present the formulations of partial algebras in the sequel.

\section{Linear Spaces with Partial Algebra Structure}
Ekhaguere \cite{Ekhaguere2007} shows the abundance of unbounded operators by re-characterizing them in terms of linear subspaces open under a defined product operation. These open and linear subspaces are the natural representations of unbounded operators. Thus, unbounded operators manifest as open linear subspaces under a product operation. Among such open linear subspaces mentioned that is of much interest for this paper is the space of unbounded linear maps on locally convex topological vector spaces, with composition of maps as product operation. Our intention is to reformulate this partial $^*$-algebra as a groupoid convolution algebra.
\begin{defx}\cite{Ekhaguere2007}
A \emph{partial algebra} is a triplet $(\mathcal{A}, \Gamma, \cdot)$, comprising a linear space $\mathcal{A}$, a partial multiplication $\cdot$ on $\mathcal{A}$, and a relation given by $$\Gamma = \{(x,y) \in \mathcal{A} \times \mathcal{A} : x \cdot y \in \mathcal{A} \}$$ such that $(x,v), (x,z), (y,z) \in \Gamma$ implies $(x, \alpha v + \beta z), (\alpha x + \beta y, z) \in \Gamma$ and then $(\alpha x + \beta y)\cdot z = \alpha(x\cdot z) + \beta(y\cdot z)$ and $x\cdot (\alpha v + \beta z) = \alpha(x\circ v) + \beta(x \circ z)$ for all $\alpha, \beta \in \mathbb{C}$.
\end{defx}
\begin{prop}
The \emph{partial algebra} $(\mathcal{A}, \Gamma, \cdot)$ corresponds to the groupoid $\mathcal{G} = (\Gamma \rightrightarrows \mathcal{A})$ with arrows defined by the relation $\Gamma = \{(x,y) \in \mathcal{A} \times \mathcal{A} : x \cdot y \in \mathcal{A} \}$.
\end{prop}
\begin{proof}
First, from the definition of the relation $\Gamma$ on $\mathcal{A}$, it follows that $\Gamma$ is a subgroupoid of groupoid of pairs $\mathcal{G} = \mathcal{A} \times \mathcal{A}$, and the maps are defined as follows. The arrows are defined by the relation $\Gamma = \mathcal{A}^{(2)} \subset \mathcal{A} \times \mathcal{A}$. \\ (i) The source and target maps $(t,s) : \Gamma \to \mathcal{A}$ are defined for any arrow $(x,y) \in \Gamma$ with $x\cdot y \in \mathcal{A}$, as $s(x,y) = y, t(x,y) = x$; \\ (ii) The objection map $o : \mathcal{A} \to \Delta \subset \Gamma$ is define as $x \mapsto (x,x)$; \\ (iii) The composition of arrows is partially defined, for $\circ : \Gamma \times \Gamma \to \Gamma,  (x,y)\circ (u,z) = (x,z)$ whenever $y = u$; \\ (iv) The inverse map is defined as $i : \Gamma \to \Gamma$ of an arrow $(x,y)$ is $(y,x)$.

Second, the groupoid satisfies the given linear conditions; for given $(x,v), (x,z), (y,z) \in \Gamma$, we have two linear subspaces defined by these arrows of the groupoid on $\mathcal{A}$ as follows $$\Gamma(x,\cdot) = \{y \in \mathcal{A} : (x,y) \in \Gamma\} \subset \mathcal{A}; \; \Gamma(\cdot,z) = \{x \in \mathcal{A} : (x,z) \in \Gamma\} \subset \mathcal{A}.$$ They are linear spaces, for given $(v,z) \in \Gamma(x,\cdot)$ and $(x,y) \in \Gamma(\cdot,z)$, it follows that $(x,\alpha v + \beta z) \in \Gamma$; which implies \[x \cdot (\alpha v + \beta z) = \alpha(x\cdot v) + \beta(x\cdot z) \in \mathcal{A} \] Thus, $\alpha(x,v) + \beta(x,z) \in \Gamma$. This shows that $\Gamma(x,\cdot)$ is a linear space.

Also, given $(\alpha x + \beta y, z) \in \Gamma(\cdot,z)$, by linearity we have $(\alpha x + \beta y)\cdot z = \alpha(x \cdot z) + \beta(y\cdot z) \in \mathcal{A}$. Thus, we have $\alpha(x,z) + \beta(y,z) \in \Gamma(\cdot,z)$ for all $\alpha, \beta \in \mathbb{K}$; showing also that $\Gamma(\cdot,z)$ is a linear space.
\end{proof}
\begin{rem}
The linear conditions imply that the subsets $\Gamma(x,\cdot), \Gamma(\cdot,z)$ of right and left \emph{multipliers} respectively, of any element of the linear space $\mathcal{A}$ are linear subspaces of $\mathcal{A}$ determined by the relation $\Gamma = \{(x,y) : x\cdot y \in \mathcal{A}\}$. In groupoid terms, every arrow $(x,y) \in \Gamma$ determines two linear subspaces $\Gamma(x,\cdot), \Gamma(\cdot,y)$ in $\mathcal{A}$. These subspaces are related to the source and target fibres as follows \[s^{-1}(y) = \Gamma_y = \{(u,y) \in \Gamma : u \in \Gamma(\cdot,y) \};\] \[t^{-1}(x) = \Gamma^x = \{(x,v) \in \Gamma : v \in \Gamma(x,\cdot)\}.\] Thus, the \emph{right multipliers} define the target fibre, while the \emph{left multipliers} define the source fibre. We extend this formulation to the definition of a partial $^*$-algebra in \cite{Ekhaguere2007} as follows.
\end{rem}
\begin{defx}
Given a partial algebra $(\mathcal{A},\Gamma, \cdot)$, a \emph{partial $^*$-algebra} or an \emph{involutive partial algebra} is a quadruplet $(\mathcal{A},\Gamma,\cdot,^*)$ such that $\mathcal{A}$ is an \emph{involutive} linear space with \emph{involution} $^*$, with $(y^*,x^*) \in \Gamma$ whenever $(x,y) \in \Gamma$ and then $(x\cdot y)^* = y^*\cdot x^*$.
\end{defx}
\begin{defx}
A \emph{partial subalgebra} (respectively \emph{partial $^*$-subalgebra}) is a subspace (respectively a $^*$-invariant subspace) $\mathcal{B}$ of $\mathcal{A}$ such that $x\cdot y \in \mathcal{B}$ whenever $x,y \in \mathcal{B}$ and $(x,y) \in \Gamma$.
\end{defx}
\begin{prop}
Given the partial $^*$-algebra $(\mathcal{A},\Gamma,\cdot,^*)$ as defined above, there is a corresponding groupoid $\Gamma \rightrightarrows \mathcal{A}$ defined by the equivalence relation $\Gamma = \{(x,y) \in \mathcal{A} \times \mathcal{A} : x\cdot y \in \mathcal{A}\}$ on $\mathcal{A}$, such that the compatibility of $\cdot$ and $^*$ in $\Gamma$ implies $(x\cdot y)^* = y^*\cdot x^*$.
\end{prop}
\begin{proof}
The proof follows immediately from that of the partial algebra, with the addition of the involutive map $^* : \mathcal{A} \to \mathcal{A}$ compatible with the partial multiplication and $\cdot : \Gamma \to \mathcal{A}$ defining the equivalence relation $\Gamma$.
\end{proof}
\begin{rem} We make the following remarks based on the formulation.\\
(1) We can define the involutive map $* : \mathcal{A} \to \mathcal{A}$ to be the inverse $i : \mathcal{A} \to \mathcal{A}.$ Then $x^* = x^{-1}$ and $(x\cdot y)^* = (x\cdot y)^{-1} = y^{-1}\cdot x^{-1} = y^*\cdot x^*$. Thus, $(x,y)^* = (y,x)$. \\
(2) The groupoid of pairs $\mathcal{G} = \mathcal{A} \times \mathcal{A}$ is also a partial algebra since it satisfies the relation. Thus, by the transitivity of the groupoid, it shows that every algebra is a partial $^*$-algebra, with inverse as involution. \\
(3) Units in the groupoid are of the form $(x,x),(y,y) \in \Gamma$, for they give identity arrows; and $(x,x)^* = (x,x)$.
\end{rem}
We extend the definition of left and right multipliers of an element above to left and right multipliers of the linear space $\mathcal{A}$ in terms of groupoid as follows.
\begin{defx}
We define the left and right \emph{multipliers} of $\mathcal{A}$ respectively as: \\  $\Gamma(\cdot,\mathcal{A}) = \{x \in \mathcal{A} : (x,u) \in \Gamma, \; \forall \; u \in \mathcal{A} \}$ \\ $\Gamma(\mathcal{A},\cdot) = \{y \in \mathcal{A} : (v,y) \in \Gamma, \; \forall \; v \in \mathcal{A}\}.$
\end{defx}
Since the source fibre $s^{-1}(y) = \Gamma_y$ is defined by the left multipliers of $y$, namely, $\Gamma(\cdot,y)$; we let $s^{-1}(y) = \Gamma(\cdot,y)$. Similarly, we let the target fibre $t^{-1}(x) = \Gamma(x,\cdot)$. Then for all $x \in \Gamma(\cdot,\mathcal{A}), y \in \Gamma(\mathcal{A},\cdot)$ respectively $\Gamma(x,\cdot) =  \Gamma(\cdot,y) = \mathcal{A}$.

From the formulation, it follows that for any other element $u \not \in \Gamma(\cdot,\mathcal{A})$ or $v \not \in \Gamma(\mathcal{A},\cdot)$, we have proper subsets $\Gamma(u,\cdot) \subset \mathcal{A}$ and $\Gamma(\cdot, v) \subset \mathcal{A}$. It is also true that $\mathcal{A}$ is invariant under the iteration of a left multiplier $x \in \Gamma(\cdot,\mathcal{A})$ (respectively a right multiplier $ y \in (\mathcal{A},\cdot)$). This gives rise to the following result on the ideal structure of the left and right multipliers.
\begin{prop}
Given the left $\Gamma(\cdot,\mathcal{A})$ and right $\Gamma(\mathcal{A},\cdot)$ multipliers of the linear space, their intersection $\Gamma(\cdot,\mathcal{A}) \cap \Gamma(\mathcal{A},\cdot)$ form an ideal of the partial $^*$-algebra $(\mathcal{A},\Gamma,\cdot,^*)$.
\end{prop}
\begin{proof}
By definition $\Gamma(\cdot,\mathcal{A})$ (respectively $\Gamma(\mathcal{A},\cdot)$) is closed (or invariant) under right (respectively left) multiplication by $\mathcal{A}$; that is, $\Gamma(\cdot,\mathcal{A}) \times \mathcal{A} \to \Gamma(\cdot,\mathcal{A})$ and $\mathcal{A} \times \Gamma(\mathcal{A},\cdot) \to \Gamma(\mathcal{A},\cdot)$. Thus, the restriction of the partial multiplication $\mathcal{A} \times \mathcal{A} \to \mathcal{A}$ to these subsets $\Gamma(\cdot,\mathcal{A}) \times \mathcal{A} \to \mathcal{A}$ and $\mathcal{A} \times \Gamma(\mathcal{A},\cdot) \to \mathcal{A}$ makes it a full multiplication or product.
\end{proof}
\begin{coro}
The left multipliers $\Gamma(\cdot,\mathcal{A})$ (respectively right $\Gamma(\mathcal{A},\cdot)$) is a left (respectively right) module of the partial $^*$-algebra $(\mathcal{A},\Gamma,\cdot,^*)$ or left (right) $\Gamma$-module.
\end{coro}
When the linear space $\mathcal{A}$ is locally convex with a Hausdorff locally convex topology $\tau$, a locally convex partial $^*$-algebra is defined by \cite{Ekhaguere2007} on $\mathcal{A}$ as follows.
\begin{defx}\cite{Ekhaguere2007}
A \emph{locally convex partial algebra} (respectively a \emph{locally convex partial $^*$-algebra}) is a quadruplet $(\mathcal{A},\Gamma,\cdot,\tau)$ (respectively a quintuplet $(\mathcal{A},\Gamma,\cdot,^*,\tau)$) comprising a partial algebra $(\mathcal{A},\Gamma,\cdot)$ (respectively a partial $^*$-algebra $(\mathcal{A},\Gamma,\cdot,^*)$) and a Hausdorff locally convex topology $\tau$ such that $(\mathcal{A},\tau)$ is a locally convex space and the maps $x \mapsto x\cdot y$ and $x \mapsto z\cdot x$ are continuous for every $y \in \Gamma(\mathcal{A},\cdot)$ and $z \in \Gamma(\cdot,\mathcal{A})$ (respectively the maps $u \mapsto u^*, x \mapsto x\cdot y$ and $x \mapsto z\cdot x$ are continuous for every $u \in \mathcal{A},\; y \in \Gamma(\mathcal{A},\cdot)$ and $z \in \Gamma(\cdot,\mathcal{A})$).
\end{defx}
To realize the locally convex partial $^*$-algebra in groupoid terms, we need the following definitions of the topological groupoid and locally convex topological groupoid. The latter is derived from a modification of the definition of locally convex Lie groupoid given in (\cite{SchmedingWockel2014}, 1.1).
\begin{defx}
A topological groupoid is a groupoid $\Gamma \rightrightarrows \mathcal{A}$ such that its set of morphisms $\Gamma$ and set of objects $\mathcal{A}$ are topological spaces, and its composition $m : \Gamma \times \Gamma \to \Gamma$, source and target $t,s : \Gamma \to \mathcal{A}$, objection $o : \mathcal{A} \to \Gamma$, and inversion $i : \Gamma \to \Gamma$ maps are continuous, with the induced topology on the set of composable arrows $\Gamma^{(2)}$ from $\Gamma \times \Gamma$.
\end{defx}
\begin{defx}\cite{SchmedingWockel2014}
Let $\mathcal{G} = (\Gamma \rightrightarrows \mathcal{A})$ be a groupoid over $\mathcal{A}$ with the source and target $t,s : \Gamma \to \mathcal{A}$ projections. Then $\mathcal{G}$ is a locally convex (and locally metrizable) topological groupoid over $\mathcal{A}$ if (i) $\mathcal{A}$ and $\Gamma$ are locally convex spaces; (ii) the topological structure of $\mathcal{G}$ makes $s$ and $t$ continuous; i.e. local projections; (iii) the partial composition $m : \Gamma \times_{s,t}\Gamma \to \Gamma$, objection $o : \mathcal{A} \to \Gamma$, and inversion $i : \Gamma \to \Gamma$ are continuous maps.
\end{defx}
\begin{prop}
The locally convex partial $^*$-algebra $(\mathcal{A},\Gamma,\cdot,^*,\tau)$ defined above gives rise to the locally convex groupoid $\Gamma \rightrightarrows \mathcal{A}$ such that the space of arrows  $\Gamma$ is given by the relation $\Gamma = \{(x,y) \in \mathcal{A} \times \mathcal{A} : x\cdot y \in \mathcal{A}\}$ on $\mathcal{A}$.
\end{prop}
\begin{proof}
From the above definitions, given that $(\mathcal{A},\tau)$ is a Hausdorff locally convex topological space; the relation $\Gamma \subset \mathcal{A} \times \mathcal{A}$ has the subspace locally convex topology induced from $\mathcal{A} \times \mathcal{A}$. This follows from the continuity of the \emph{partial multiplication} $(\cdot)$ defining the relation $\Gamma$ which preserves local convexity. Also, the continuity of involution $^*$ is by definition; and the defining maps of the groupoid; the target and source maps $t,s : \Gamma \to \mathcal{A}$, the inverse $i : \Gamma \to \Gamma$, and the composition of arrows $m : \Gamma \times \Gamma \to \Gamma$ are all continuous maps. Thus, $\Gamma \rightrightarrows \mathcal{A}$ is a locally convex topological groupoid representing the locally convex partial $^*$-algebra $(\mathcal{A},\Gamma,\cdot,^*,\tau)$.
\end{proof}
In the following section we show that the locally convex topological groupoid $\Gamma \rightrightarrows \mathcal{A}$ is a Lie groupoid modelled on the locally convex topological vector space $\mathcal{A}$.

\section{The Lie Groupoid Modelled on $(\mathcal{A},\tau)$}
According to Hideki Omokri \cite{OmokriHideki97}, a Lie group is a group in which the infinitesimal neighbourhood of the identity element generates the connected component of the group containing the identity element. Since we are dealing with transformations of a linear space, the transformation generating the connected component (the Lie group) can be modified in such a way that any fixed element of the space being transformed can give the identity transformation. In the light of this possibility, we can modify the above description of a Lie group as a group in which the infinitesimal neighbourhood of any element generates the connected component of the group of transformations giving an identity transformation on the fixed element.

We are interested in such (connected) components $\mathscr{K}$ of a transformation group which is generated on an open subspace that is dense in the original locally convex linear space $\mathcal{A}$. The corresponding infinitesimals are properly the unbounded operators on the linear space. Their algebras are the partial and partial $^*$-algebras. Hence, we have to consider the open subspace of right multipliers or the target fibre $\Gamma(x,-)$ (or the source fibre $\Gamma(-,x)$) of the groupoid we have constructed from the relation $\Gamma$ on $\mathcal{A}$. As we have noted above, these subspaces are maximal when $x \in \Gamma(-,\mathcal{A})$ (respectively $x \in \Gamma(\mathcal{A},-)$.

Furthermore, 'to generate' a component, as noted by Omokri (\cite{OmokriHideki97}, p.1), implies various means of 'integrating' an infinitesimal quantity to give a finite quantity. These various means may be solving an ordinary differential equation, solving a partial differential equation of evolution, product integral, or Feynman path integral. The task therefore is to define the infinitesimal neighbourhood of identity (or of any element) in the transformations of the locally convex Hausdorff topological space $\mathcal{A}$ defined by the partial multiplication ($\cdot$) giving rise to the relation $\Gamma$.

Since the space is locally convex, we employ Alain Connes' technique for a smooth manifold as in \cite{Connes94}; whereby every point $x$ in the manifold $M$ is identified with a convex neighbourhood $(x,\varepsilon)$, where $\varepsilon \in [0,1)$. The convex neighbourhood $(x,\varepsilon) \in M \times [0,1)$ is then used to relate the points to the infinitesimal generators acting at $x$, as element of the tangent space $T_xM$ and the tangent bundle $TM$.

In this case, the pair $(x,\varepsilon) \in \mathcal{A}\times (0,1)$ can be said to represent an open convex neighbourhood of each point $x \in \mathcal{A}$. This follows because the convexity of a set $A \subset \mathcal{A}$ implies that it is invariant under a homothetic transformation centred at any point $a \in A$ with ratio $\varepsilon \in (0,1)$. (cf. \cite{Bourbaki81},II,$\S$ 2). Hence, by local convexity of $\mathcal{A}$, we have that $(x,\varepsilon) \in \mathcal{A} \times [0,1)$ encodes open convex neighbourhoods of every point $x$ in $\mathcal{A}$.

Next, we extended this construction to the groupoid $\Gamma \rightrightarrows \mathcal{A}$ determined by the relation $\Gamma = \{(x,y) \in \mathcal{A} \times \mathcal{A} : x\cdot y \in \mathcal{A}\}$. The set of units which is the image of the objection map $o : \mathcal{A} \to \Gamma^o = \{(x,x) \in \Delta \subset \mathcal{A} \times \mathcal{A}\} \simeq \mathcal{A}$ is locally convex. The local convexity of $\mathcal{A}$ and the fact that $(\mathcal{A},\Gamma,\cdot,^*,\tau)$ is a locally convex partial $^*$-algebra assures $(x,x) \in \Gamma, \; \forall \; x \in \mathcal{A}$, since according to \cite{DierolfHeintz}, a locally convex topological space is an algebra if and only if for any $0$-neighbourhood $U$ in $(\mathcal{A},\tau)$ there is another $0$-neighbourhood $V$ satisfying $V^2 = \{x\cdot y : x,y \in V\}\subset U$. Equivalently, there is a $0$-neighbourhood filter in $(\mathcal{A},\tau)$ with a basis consisting of sets that are stable with respect to multiplication; that is, for a $0$-neighbourhood $U$, there is stable $0$-neighbourhood $V$ satisfying $V^2 \subset V \subset U$.

These imply the compatibility of the algebraic and topological structures in $(\mathcal{A},\Gamma,\cdot,^*,\tau)$, at least in the convex neighbourhood of its points. Thus, the objection map $x \mapsto (x,x) \in \Gamma$ is defined since $x\cdot x \in \mathcal{A}$ and the set of arrows $\Gamma \subset \mathcal{A} \times \mathcal{A}$ is also locally convex since it is determined by the (partial) multiplication. Given the partial multiplication which presupposes compatibility of the algebraic and topological structures, we define a smooth structure on the algebra as follows.

Given the local convexity of $\mathcal{A}$ there exists a locally convex neighbourhood $(x,\varepsilon)$ of a point $x \in \mathcal{A}$ such that we can define a smooth  or Lie groupoid by employing Connes' representation of the locally convex neighbourhoods of the points of $\mathcal{A}$ as $(x,\varepsilon), \varepsilon \in [0,1)$. This is also used to define the locally convex neighbourhood of an arrows $(x,y) \in \Gamma$ as $(x,y,\varepsilon) \in \Gamma \times [0,1) \subset \mathcal{A} \times \mathcal{A} \times [0,1)$. The composition of these infinitesimal arrows becomes \[ (x,y,\varepsilon)\circ (y,z,\varepsilon) = (x,z,\varepsilon) \; \text{for } \varepsilon \in [0,1), x,y,z \in \mathcal{A}. \] This defines a smooth structure on $\mathcal{G} := \Gamma \times [0,1) \rightrightarrows \mathcal{A} \times [0,1)$, thereby making it a smooth groupoid. We give this as a result which is proved as follows.

\begin{prop}
The map $T\mathcal{A} \times [0,1) \rightrightarrows \mathcal{A} \to \Gamma \times [0,1) \rightrightarrows \mathcal{A} \times [0,1)$, defined on the arrows $\Gamma(-,x) \times [0,1) \to T_x\mathcal{A}$ is a groupoid isomorphism.
\end{prop}
\begin{proof}
First, the tangent bundle $T\mathcal{A}$ is a groupoid $T\mathcal{A} \rightrightarrows \mathcal{A}$ which is a union of groups $T_x\mathcal{A}$, with $x \in \mathcal{A}$ as the identity element, such that the objection map is $\mathcal{A} \ni x \mapsto (x,0) \in T_x\mathcal{A}$. So the set of arrows of the groupoid $ T\mathcal{A} \rightrightarrows \mathcal{A}$ is $T\mathcal{A} = \{(x,X) : x \in \mathcal{A}\}$, where $X$ is a vector field or infinitesimal generator at the point $x \in \mathcal{A}$ and $X|_x$ is a tangent vector to $\mathcal{A}$ at $x$. The target and source maps for an arrow are defined $s(x,X) = (x,0)$, $t(x,X) = (x,0)$; the composition of vector fields at $x$ is given as $(x,X_1)\circ (x,X_2) = (x,X_1 + X_2)$.

The equivalence of these groupoids is now established as follows. By definition, the smooth or Lie groupoid $\Gamma \times [0,1) \rightrightarrows \mathcal{A} \times [0,1)$ has an open convex neighbhourhood of the identity arrow $(x,x)$ as $(x, x, \varepsilon)$ which is constructed from the locally convex partial $^*$-algebra. From the locally convex infinitesimal neighbourhood of the generators $(x,X,\varepsilon)$ of the tangent space to $\mathcal{A}$ at $x$, contained in the groupoid $T\mathcal{A} \times [0,1) \rightrightarrows \mathcal{A}$, the exponential map generates the open submanifold $\Gamma(-,x)$ of the Lie groupoid as follows \[\exp : T\mathcal{A} \times [0,1) \to \mathcal{G}(-,x) \times [0,1); \; (x, X, \varepsilon) \mapsto (x,\exp_x(-\varepsilon X),\varepsilon) \] where $(x,X,0) \mapsto (x,x) \in \Delta \subset \mathcal{A} \times \mathcal{A} \simeq \mathcal{A}$ is an identity.
\end{proof}
\begin{coro}
The locally convex groupoid $\mathcal{G} := \Gamma \times [0,1) \rightrightarrows \mathcal{A} \times [0,1)$ generated from the locally convex infinitesimal neighbourhoods of $T\mathcal{A} \rightrightarrows \mathcal{A}$ is a Lie groupoid.
\end{coro}
\begin{proof}
The proof of this follows immediately from the proof of the preceding result. Reference can also be made to Connes' construction in (\cite{Connes94},Ch.2, Sect.5).
\end{proof}
\begin{lem}
The connected component $\mathscr{K}$ of arrows generated by the locally convex neighbourhood of an identity arrow $(x,x,\varepsilon)$, which corresponds to the locally convex neighbourhood $(x,\varepsilon)$ of the point $x \in \mathcal{A}$, has the action of $[0,1)$. It is therefore isomorphic to the $\mathcal{G}(x,x)$-space $\Gamma(x,-)$ which is a subspace of Lie groupoid $\Gamma \rightrightarrows \mathcal{A}$.
\end{lem}
\begin{proof}
Since the smooth locally convex open submanifold $\Gamma(x,-)$ of the Lie groupoid $\Gamma \rightrightarrows \mathcal{A}$ is generated by the locally convex neighbourhood $(x,\varepsilon)$ of $x \in \mathcal{A}$ (or the neighbourhood $(x,x,\varepsilon)$ of the identity arrow $(x,x)$), it is a connected open subspace of $\mathcal{A} \times \mathcal{A} \times [0,1)$ according to \cite{OmokriHideki97}. This is equivalently expressed as $\mathcal{A} \times [0,1) \to \mathcal{A}$. For $I = [0,1)$, this defines a map $I \to \mathcal{A}^\mathcal{A}, \varepsilon \mapsto \gamma_\varepsilon = (x,y)$ which is action of a net on the arrows, such that the arrows $\{(x,y) : y \in \mathcal{A}\}$ converge to the identity arrows $(x,x)$ as $\varepsilon \to 0$. This action defines a net of local bisections (to be see in the subsequent sections). The nature of the open locally convex submanifold of arrows $\Gamma(x,-) \subset \Gamma$ depends on each point $x \in \mathcal{A}$. It is maximal when $x \in \Gamma(-,\mathcal{A})$-a left multiplier of $\mathcal{A}$.

Now, let $\mathscr{K}$ be the connected component of arrows generated by $(x,x,\varepsilon)$ or $(x,\varepsilon)$. It is a smooth net resulting from the smoothing action of $[0,1)$ on $\Gamma(x,-)$. The partial symmetry of $\Gamma(x,-)$ is encoded by the net $\mathscr{K}$. Since $(x,\varepsilon y) \in \Gamma(x,-), \varepsilon \in [0,1)$, it follows that the action $\Gamma(x,-) \times \mathscr{K} \to \Gamma(x,-)$ encodes the sectional dynamical system of the Lie groupoid $\Gamma$. Thus, $[0,1) \simeq \mathscr{K}$. When the endpoints of the interval $[0,1)$ are identified, the net action is equivalent or similar to the action of the isotropy Lie group $\mathcal{G}(x,x)$ on $\Gamma(x,-)$.
\end{proof}
The inverse operation to the generation of the smooth connected component $\mathscr{K}$ is the limit process on the derived filters $\{\mathcal{F} \to x\}$ of the net action of $[0,1)$. Thus, the filters correspond to the convergent nets $x_\varepsilon \to x, y_\varepsilon \to x$, where $(x_\varepsilon, y_\varepsilon) \to (x,x) \in \Gamma$ and $(x, y_\varepsilon) \in \Gamma(x,-)$. This implies that the net of arrows $(x_\varepsilon, x)$ and $(x,y_\varepsilon)$ are generated by the smooth net $\mathscr{K}$ on a given arrow $(x,y)$ making up the locally convex linear submanifolds $\Gamma(x,-), \Gamma(-,x)$.

If we now consider the infinitesimal neighbourhood, it will follow that the convergence of the two nets gives $\mathcal{G}(x,x) = \Gamma(x,-) \times [0,1) \cap [0,1)\times \Gamma(-,x)$; while the filter $\mathcal{F} \to x$ is equivalent to $\Gamma(x,-) \times [0,1) \cup [0,1)\times \Gamma(-,x) \to \mathcal{G}(x,x)$ which is the isotropy Lie group of the smooth groupoid $\mathcal{G} = (\Gamma \rightrightarrows \mathcal{A})$. This is equivariant to the convergence of the quotient \[(x_\varepsilon \to x, y_\varepsilon \to x, \frac{x_\varepsilon - y_\varepsilon}{\varepsilon} \to X, \varepsilon \to 0) \to (x,X,0), \] for each net of arrows in $\Gamma(x,-) \times [0,1)$. This is a net condition for smoothness similar to Connes' sequential condition because of the smooth $\mathscr{K}$-action on $\mathcal{G}(x,-)$. (Cf. \cite{Connes94}).
\begin{rem}
We make the following remarks. \\ First, the vector fields $X \in T_x\mathcal{G}(x,x)$ correspond to the universal flows $\varphi(\tau,x)$ associated with the Lie groupoid $\Gamma \rightrightarrows \mathcal{A}$, where $\varphi(0,x) = x$ for any fixed $x \in \mathcal{A}$.  This implies $X|_x \in  T_x(\Gamma(x,-)) \simeq T_x(\Gamma(-,x))$. \\
Second, there are also local fields $Y|_x \in T_x(\Gamma(x,-))$ defined for some $x \in \mathcal{A}$ and not for all; that is, $Y|_y = 0$ for some $y \in \mathcal{A}$; these are not parallelizable (or globalizable through parallel transport). The maximal of these tangent spaces occurs when $x \in \Gamma(-,\mathcal{A})$ (or $\Gamma(\mathcal{A},-)$). Because the subspace generated by the convex neighbourhood of such $x \in \mathcal{A}$ is maximal and dense in $\mathcal{A}$, the corresponding flows are connected to the unbounded infinitesimal generators of the Lie groupoid $\Gamma \rightrightarrows \mathcal{A}$. They give rise to unbounded operators as described in the introduction of \cite{Antoine-etal2002}.\\
Third, the closed submanifolds $\Gamma(x,y) \subset \Gamma$ are $\mathcal{G}(x,x)$-spaces of the same dimension with $\mathcal{G}(x,x)$, while the open (maximal) submanifold $\Gamma(-,x) \simeq \Gamma(x,-)$ are $\Gamma$-modules and infinite dimensional. The structure and the symmetry of the open and dense submanifolds are now explored with the local bisections $B_\ell(\mathcal{G})$ which are Lie pseudogroup characterized as follows.
\end{rem}
\begin{defx}
Given the isomorphism $\mathscr{B}(\Gamma) \simeq \text{Diff}(\mathcal{A})$, defined by $\varphi \leftrightarrow t \circ \varphi$, between the set of diffeomorphisms of the base manifold of a Lie groupoid and the set of bisections of the Lie groupoid in \cite{Mackenzie2005}, which can be restricted to the collection of the local diffeomorphisms $\text{Diff}(U)$ of a trivialization $U$ of the base manifold and the local bisections defined on it $\mathscr{B}_\ell(U)$, we have the following properties defining a Lie pseudogroup: \\ (1) For any $t \circ \varphi \in \text{Diff}(U)$, $t\circ \varphi : U \to \mathcal{A} \; \implies \; t \circ \varphi|_V \in \text{Diff}(U)$, for all $V \subseteq U$; \\ (2) If $L \subseteq \mathcal{A}$ is open with $L = \underset{\alpha}\bigcup U_\alpha$, then if $t \circ \varphi : L \to \mathcal{A}$ with $t \circ \varphi|_{U_\alpha} \in \text{Diff}(L)$, then $t \circ \varphi \in \text{Diff}(L)$; \\ (3) $\text{Diff}(U)$ is closed under composition, since for any two bisections $\sigma$ and $\tau$, the homomorphism property of the target map $t$ implies $t : \sigma \star \tau \mapsto t \circ (\sigma \star \tau) = (t \circ \sigma) \circ (t \circ \tau)$; (closure of composition implies this, since $\sigma \star \tau$ is defined whenever a diffeomorphism $\phi$ links $\sigma$ and $\tau$.) \\ (4) The identity diffeomorphisms are in $\text{Diff}(U)$; \\ (5) Each $t \circ \varphi$ has an inverse $(t \circ \varphi)^{-1}$. This follows because of the localization, say at the neighbourhood of $y$, where $(t_y \circ \varphi)^{-1} = \varphi^{-1} \circ t_y^{-1}$; thus, given any $x$ in the neighbourhood of $y$, i.e. in $U$, we have $\varphi^{-1} \circ t_y^{-1}(x) = \varphi^{-1}(\Gamma(x,y)) = y$.
\end{defx}
\begin{rem}
We note the fact that the smoothing action of $[0,1)$ has made the locally convex space $\mathcal{A}$ which is the base space of the Lie groupoid a smooth manifold. Hence the objection map $o : x \mapsto (x,x)$ is an embedding into a smooth manifold. So, the set $\mathscr{B}_\ell(U)$ of local bisections of a trivialization $U$ is a Lie pseudogroup.
\end{rem}

We have successfully modelled the locally convex partial $^*$-algebra $(\mathcal{A},\Gamma,\cdot,^*,\tau)$ as a locally convex Lie groupoid $\Gamma \rightrightarrows \mathcal{A}$ which is a smooth locally convex manifold modelled on the locally convex topological space $\mathcal{A}$ given the relation $\Gamma$. This formulation gives the isomorphism $\Gamma \rightrightarrows \mathcal{A} \simeq C^\infty(\mathcal{A})|_\Gamma$. The following definitions of a smooth manifold modelled on a locally convex topological space and the diffeomorphisms arising from the model apply to the formulation.
\begin{defx}\cite{Glockner2006}
A smooth manifold $M$ modelled on a locally convex topological space $\mathcal{A}$ is a Hausdorff topological space, together with a set $\mathscr{G}$ of homeomorphisms from open subsets of $M$ onto open subsets of $\mathcal{A}$, such that the domains cover $M$ and the transition maps are smooth.
\end{defx}
\begin{defx}\cite{Glockner2006}
Let $E$ and $F$ be real locally convex spaces, $U \subseteq E$ be open, and $f : U \to F$ be a map. For $x \in U$ and $y \in E$, let $(D_yf)(x) := \frac{d}{dt}|_{t=0}f(x + ty)$ be the directional derivative (if it exists). Given $k \in \mathbb{N}\cup\{\infty\}$, the map $f$ is called $C^k$ if it is continuous, the iterated directional derivatives $d^jf(x,y_1, \cdots, y_j) := D_{y_j}\cdots D_{y_1}f(x)$ exist for all $j \in \mathbb{N}$ such that $j \leq k, x \in U$ and $y_1, \cdots, y_j \in E$, and all of the maps $d^jf : U \times E^j \to F$ are continuous. ($f$ is smooth if $k = \infty$.)
\end{defx}
As hinted in the introduction, the existence of a Lie group structure on the set of bisections of the Lie groupoid $\Gamma \rightrightarrows \mathcal{A}$, which is related and derivable from the canonical smooth structure on the manifold of arrows $\Gamma$, is a result of the local convexity of $\mathcal{A}$. This understanding that the Lie groupoid $\Gamma \rightrightarrows \mathcal{A}$ is a smooth manifold modelled on a locally convex topological space $\mathcal{A}$ carrying a partial $^*$-algebra given by the relation $\Gamma$ is subsequently employed to define its smooth system of Haar measures.

\section{Haar System and Inductive Limit of Lie Groupoid}
The diffeomorphims and their differentials constitute differential forms on a smooth manifold. According to \cite{Morita97}, differential forms are used to express various geometric structures on manofolds.To obtain certain geometric "invariants", appropriate operations are applied to differential forms which are integrable on manifolds. We briefly describe the operations to be applied to differential forms to obtain Haar measures for our Lie groupoid.

We first propose that unbounded operators correspond to open and dense locally convex subspaces $\Gamma(x,-), \Gamma(-,y)$. We shall also linke these to locally convex neighbourhoods and their corresponding infinitesimal generators. There is need to connect these to differential forms which are integrable objects of a (smooth) manifold. Since we have already seen that $\Gamma(x,-)$ and $\Gamma(-,y)$ are locally convex topological spaces given the Lie groupoid $\Gamma \rightrightarrows \mathcal{A}$ defined by the relation $\Gamma = \{(x,y) : x\cdot y \in \mathcal{A}\}$ on $\mathcal{A}$, we consider the following smooth functions associated with the parameterizaton of the open submanifold of arrows $\Gamma(x,-)$. \[ \phi(x + ty) : \Gamma(x,-) \to \R^m \eqno{(1)}\] where $m$ is the dimension of the maximal closed neighbourhood $\Gamma(x,y) \subset \Gamma(x,-)$ containing a trivialization $x$. These functions constitute the set of homeomorphisms from open subsets of $\Gamma(x,-)$ of the Lie groupoid onto open subsets of $\R^m$. From this background, we employ the definition of smooth measures on a smooth manifold as given in \cite{Folland99}.
\begin{defx}
Let $M$ be a smooth manifold of dimension $m$. Then a smooth measure on $M$ is a Borel measure $\mu$ which is given in local coordinate $x$ as $d\mu = \phi^xdx$, where $\phi^x$ is a nonnegative smooth function. The change of coordinates for the smooth measure $\mu$ is given as \[ \phi^x = |\det\left(\frac{\partial y}{\partial x}\right)|\phi^y \; \implies \; d\mu = |\det\left(\frac{\partial y}{\partial x}\right)|\phi^ydy \eqno{(2)} \]
\end{defx}
This is also related to density on the smooth manifold $\Gamma \rightrightarrows \mathcal{A}$, which is a section of the smoothly varying line bundle on the Lie groupoid. From the above, the Borel measures on $\Gamma \rightrightarrows \mathcal{A}$ are differential forms generated or spanned by $dx$ in the coordinate system $x$ of $\Gamma(x,y) \subset \Gamma(x,-)$. It follows that the measures vary according to the smooth function $\phi^x$; so that $\phi^x$ represents a density (or its element) in each coordinate system $x$, and its variation in coordinate change is governed by (2). This helps us to define both a differential form and a density on $\Gamma \rightrightarrows \mathcal{A}$ as follows.
\begin{defx}
A differential form or an $m$-form on an $m$-dimensional manifold $\Gamma(x,y) \subset \Gamma(x,-)$ is a section of the line bundle whose transition functions are $\det(\frac{\partial y}{\partial x})$. So, given an $m$-form $\omega^x$ for the coordinate system $x$ or $(U,\phi^x)$ in $\Gamma(x,-)$, its transformation in the $y$ coordinate system of $\Gamma(y,-)$ is $\omega^y$; and the two are related by the formula \[ \omega^x = \det \left(\frac{\partial y}{\partial x}\right)\omega^y. \eqno{(3)}\]
The usual notation for the differential form is $\omega = \omega^x dx_1\wedge \cdots \wedge dx_m$.
\end{defx}
\begin{defx}
A density is a function on coordinate charts of $\Gamma(x,-)$ which becomes multiplied by the absolute value of the Jacobian determinant in the change of coordinates $\phi^x = |\det(\frac{\partial y}{\partial x})|\phi^y$ or between the coordinate charts on the Lie groupoid $\Gamma \rightrightarrows \mathcal{A}$.
\end{defx}
Thus, orientation is the only difference between a density $\phi^x$ and a differential form $\omega^x$; the former is without an orientation, while the latter is with orientation. The two coincide when there is a restriction to a coordinate system with Jocabian matrices of positive determinant. So, the differential forms which are smooth measures are identified with nonnegative densities on the smooth manifold of arrows/maps of the Lie groupoid $\Gamma \rightrightarrows \mathcal{A}$.

Based on the treatment of \cite{Folland99}, we see that any density $\phi$ (a diffeomorphism) on the manifold of arrows $\Gamma$ of the Lie groupoid $\Gamma  \rightrightarrows \mathcal{A}$ defines, at least locally,  a smooth signed (or complex) measure $\mu$ on $\Gamma \rightrightarrows \mathcal{A}$, so that the following integrals are well defined for any compact set $K \subset \Gamma$;  \[ \int_K \phi = \mu(K); \; \text{and } \int f\phi = \int fd\mu; \eqno{(4)} \] for any $f \in C_c(\Gamma \rightrightarrows \mathcal{A}) = C_c(\Gamma)$. So, the diffeomorphisms defined in (1) above are also measures supported on compact (closed) subsets $\Gamma(x,y)$ of $\Gamma(x,-)$.

When the densities (as measures) are normalized, they give rise to probability measures as given in \cite{Folland99}. This is done by having $0 \leq \theta \leq 1$, which presents a $\theta$-density on the Lie groupoid $\Gamma \rightrightarrows \mathcal{A}$ to be a section of the line bundle whose transition functions on $\Gamma(x,-) \cap \Gamma(y,-)$ are $|\det \left(\frac{\partial y}{\partial x}\right)|^\theta$. Given this normalization, a $1$-density becomes a standard density while a $0$-density is a smooth function on $\Gamma \rightrightarrows \mathcal{A}$. When $0 < \theta < 1$ and $p = \theta^{-1}$ and $\phi$ is a $\theta$-density, then $|\phi|^p$ is well defined density and therefore integrable over the Lie groupoid $\Gamma \rightrightarrows \mathcal{A}$.

Alternatively, since densities are also smooth functions on the Lie groupoid $\Gamma \rightrightarrows \mathcal{A}$, the set of $\theta$-densities $\phi$ (usually denoted $|\Omega|^p(\Gamma)$, with $p = \theta^{-1}$) satisfying the norm condition \[ ||\phi||_p = \left(\int |\phi|^p \right)^{\frac{1}{p}} < \infty \] is a normed linear space with $L^p(\Gamma)$ as its completion. In other words, $|\Omega|^p(\Gamma) \subset L^p(\Gamma)$. The completion is the \emph{intrinsic $L^p$-space} of the normed space $|\Omega|^p(\Gamma)$ connected with the Lie groupoid $\Gamma \rightrightarrows \mathcal{A}$.

By the duality associated to this definition of the space of densities on the Lie groupoid $\Gamma \rightrightarrows \mathcal{A}$ as a dense subset of $L^p(\Gamma)$ which also follows from \cite{Folland99}, the product of two $\frac{1}{2}$-densities is 1-density. Thus, an inner product is defined on the space of $\frac{1}{2}$-density with compact support as \[ \langle \omega, \eta \rangle = \int_\Gamma \omega\bar{\eta}. \eqno{(5)} \] This makes the space of $\frac{1}{2}$-densities $|\Omega|^2(\Gamma) \subset L^2(\Gamma)$ a pre-Hilbert space, with $L^2(\Gamma)$ as its canonically associated Hilbert space completion. With this understanding, the definition of smooth system of Haar measures on a Lie groupoid follows.

\subsection{System of Haar Measures and Convolution Algebra}
Since we have shown the groupoid model of the locally convex partial $^*$-algebra $(\mathcal{A},\Gamma, \cdot,^*, \tau,)$ to be a Lie groupoid, and the diffeomorphisms $\phi(x + ty) : \Gamma(x,-) \to \R^m$ show that $\Gamma(x,-)$ is locally compact; it follows that it has a left Haar system of measures defined by Paterson as follows.
\begin{defx}(cf. \cite{Paterson99})
A left Haar system for the Lie groupoid $\mathcal{G} := \Gamma \rightrightarrows \mathcal{A}$ is a family $\{\mu^x \}_{x \in \mathcal{A}}$, where each $\mu^x$ is a positive regular Borel measure on the locally compact (convex) Hausdorff space $\Gamma(x,-)$, such that the following three axioms are satisfied.\\ (i) the support of each $\mu^x$ is the $t$-fibre $\Gamma(x,-)$; \\
(ii) for any $f \in C_c(\Gamma)$, the function $\displaystyle f_o(x) = \int_{\Gamma(x,-)}fd\mu^x$ belongs to $C_c(\mathcal{A})$; \\
(iii) for any $\gamma \in \Gamma$ and $f \in C_c(\Gamma)$, $\displaystyle \int_{\Gamma(s(\gamma),-)}f(\gamma\eta)d\mu^{s(\gamma)}(\eta) = \int_{\Gamma(t(\gamma),-)}f(\kappa)d\mu^{t(\gamma)}(\kappa)$.
\end{defx}
This is in line with the smoothness condition for a Lie groupoid and our consideration above; it is also related to a choice of appropriate local coordinates making the Radon-Nikodym derivatives of the family $\{ \mu^x \}$ strictly positive and smooth. Hence, the smooth left Haar system of measures on the Lie groupoid $\Gamma \rightrightarrows \mathcal{A}$ is unique up to equivalence due to the smooth net $\mathscr{K}$-action; which means they are in the same class of measures. Thus, the system is isomorphic to the strictly positive sections of the 1-density line bundle $\Omega^1(T_\gamma(\Gamma(x,-)^*)$. This is established as follows.

By the bijection between (cross) sections of a fibre bundle and the set of maps from a base space to the fibres, a section $\varphi$ uniquely defines a function from the base space to the fibre $f : \mathcal{A} \to \Gamma$. Thus, a section $\varphi$ is of the form $\varphi(x) = (x,f(x)), x \in \mathcal{A}, (x,f(x)) \in \Gamma \subset \mathcal{A} \times \mathcal{A}$. (Cf. \cite{Husemoller94}). By definition of the partial product, the definition of the function $f : \mathcal{A} \to \Gamma$, and subsequently, the section $\varphi$ is localized. Thus, we have $\varphi(x) = f(x + \varepsilon y)$ by the smooth action of $[0,1)$, where $\varepsilon \in [0,1)$, giving rise to differential form as given above.

Furthermore, though $\varphi$ is a section to the source map $s : \Gamma \to \mathcal{A}$, which means $s \circ \phi = I_\mathcal{A}$, it is required to define a diffeomorphism with the target map; that is, $t \circ \varphi : \mathcal{A} \to \mathcal{A}$ for it to be a bisection of the Lie groupoid $\Gamma \rightrightarrows \mathcal{A}$. The partial product forces it to be a local bisection $\varphi \in \mathcal{B}_\ell(\mathcal{G}) \subset \mathcal{B}(\mathcal{G})$ by restricting to an open neighbourhood $t(N_x) = t((x,\varepsilon)) \subseteq \mathcal{A}$ of $x \in \mathcal{A}$ where the map $\varphi \mapsto t \circ \varphi$ is defined, the target map $t_x : \Gamma(-,x) \to \mathcal{A}$ is a surjective submersion, and $t\circ \varphi : N_x \to (t\circ \varphi)(N_x)$ a diffeomorphism. (Cf. \cite{Mackenzie2005}, Definition 1.4.8).

In addition, since the set of (cross) sections of a fibre bundle constitutes a module over the ring of continuous functions from the base space to the fibres $\mathcal{A} \to \Gamma$, which are manifold-valued functions, it follows that the local bisections $\mathcal{B}_\ell(\mathcal{G})$ constitute a module over the arrows which are continuous functions uniquely determined by the local bisections, and a Lie pseudogroup by the smooth action of $[0,1)$. Thus, by the bijection $\varphi \mapsto t \circ \varphi, \mathcal{B}_\ell(\mathcal{G}) \to \text{Diff}(N_x)$ which preserves the partition of $\mathcal{A}$ under the relation $\Gamma$, we can replace the arrows $\gamma \in \Gamma$ with the local bisections $\varphi \in \mathcal{B}_\ell(\mathcal{G})$. Thus, the partial product structure of $\Gamma \rightrightarrows \mathcal{A}$ is encoded by the local bisections $\mathcal{B}_\ell(\mathcal{G}) \subset \mathcal{B}(\mathcal{G})$.

The smooth chart defined by (1) lays bare the structure of this manifold of sections (or arrows), and also defines Borel measures on it as given above. Thus, the use of the open (convex) neighbourhood $N_x \simeq (x,x,\varepsilon) \subset \Gamma(x,-)$ as a $t$-fibrewise product agrees with the definition of the smooth Haar system using the diffeomorphisms of (1), which are also connected to the bisections $\mathcal{B}_\ell(\mathcal{G})$.
\begin{defx}(cf. \cite{Paterson99})
Let $N_x$ be an open subset of the Lie groupoid $\mathcal{G} := \Gamma \rightrightarrows \mathcal{A}$. Since $t : \Gamma(x,-) \to \mathcal{A}$ is a submersion, it is open. So, $t(N_x)$ is open in $\mathcal{A}$. The pair $(N_x,\phi)$ is called a $t$-fibrewise product if there exists an open subset $W$ of $\R^m$ containing $0$, and $\phi$ is a diffeomorphism from $N_x$ onto $t(N_x) \times W$ preserving $t$-fibres in the sense that $p_1(\phi(\gamma)) = t(\gamma), \; \forall \; \gamma \in N_x$, where $p_1$ is the projection on the first coordinate of $t(N_x) \times W$.
\end{defx}
\begin{defx}
A smooth left Haar system for the Lie groupoid $\Gamma \rightrightarrows \mathcal{A}$ is a family $\{\mu^x\}_{x \in \mathcal{A}}$ where each $\mu^x$ is a positive, regular Borel measure on the submanifold $\Gamma(x,-)$ such that: \\ (i) If $(N_x,\phi)$ is a $t$-fibrewise product open subset of $\mathcal{G}$, $N_x \simeq t(N_x) \times W$, and if $\mu_W = \mu|_W$ is Lebesgue measure on $\R^m$, then for each $x \in t(N_x)$, the measure $\mu^x \circ \phi^x$ is equivalent to $\mu_W$, since $\phi^x : N_x\cap \Gamma(x,-) \to \R^m$ is a diffeomorphism and their R-N derivative is the function $\Phi(x,w) = d(\mu^x \circ \phi^x)/d\mu_W(w)$ belonging to $C^\infty(t(N_x) \times W)$ and is strictly positive.\\
(ii) With $\Phi|_{\mathcal{A}}$, we have $\displaystyle \Phi_o(x) = \int_{\Gamma(x,-)}\Phi d\mu^x$, which belongs to $C_c(\mathcal{A})$.\\
(iii) For any $\gamma \in \Gamma$ and $f \in C^\infty_c(\Gamma)$, we have \[ \int_{\Gamma(s(\gamma),-)} f(\gamma\eta)d\mu^{s(\gamma)}(\eta) = \int_{\Gamma(t(\gamma),-)}f(\xi)d\mu^{t(\gamma)}(\xi) \]
\end{defx}
This definition makes a clear sense in the light of the smooth structure on $\mathcal{B}_\ell(\mathcal{G})$ subsequent on the smooth structure of $\Gamma \rightrightarrows \mathcal{A}$, for it follows the definition of the smooth chart $(N_x,\phi^x) \simeq W \subset \R^m$. So, for each $x \in N_x$, $\mu^x$ is positive and a smooth measure on $N_x\cap \Gamma(x,\cdot)$; and since $\phi^x : N_x \to W \subset \R^m$ is a diffeomorphism, $\mu^x \circ \phi^x \sim \mu_W$. Thus, the Radon-Nikodym derivative $\frac{d(\mu^x \circ \phi^x)}{d\mu_W}$ varies smoothly on $\Gamma \rightrightarrows \mathcal{A}$ by definition.

With these formulations, and given the normalized $\frac{1}{2}$-densities, where $|\Omega|^{1/2}_\gamma$ is the fibre over an arrow $\gamma \in \Gamma$  with $t(\gamma) = x, s(\gamma) = y$; a density $\phi \in C^\infty_c(\Gamma,\Omega(\Gamma))$  determines a functional $\Omega^kT_\gamma(\Gamma(x,\cdot))\otimes \Omega^kT_\gamma(\Gamma(\cdot,y)) \to \R$. The convolution algebra of the Lie groupoid $\Gamma \rightrightarrows \mathcal{A}$ is therefore defined on the space of sections (densities) $C^\infty_c(\Gamma,\Omega(\Gamma)) \subset L^2(\Gamma)$ of the line bundle, with the convolution product $f * g$ of $f,g \in C^\infty_c(\Gamma,\Omega(\Gamma))$ given as \[ f * g(\gamma) = \int_{\eta\circ \xi = \gamma}f(\eta)g(\xi) = \int_{\Gamma(t(\gamma),-)}f(\eta)g(\eta^{-1}\gamma). \eqno{(6)}\] The involution is defined as \[ f^*(\gamma) = f(\gamma^{-1}). \eqno{(7)}\]
The integral is that of sections on the manifold $\Gamma(t(\gamma),-)$ since $f(\eta)g(\eta^{-1}\gamma)$ is a 1-density. Based on Paterson's reformulation (cf. \cite{Paterson99}, Appendix F), the above is alternatively given as \[f * g(\gamma) = (\int \omega)(\omega_{s(\gamma)}\otimes \omega_{t(\gamma)}),  \eqno{(8)} \] with $f, g, f*g \in C_c^\infty(\Gamma, \Omega(\Gamma))$.

As noted above, the definition of $1/2$-densities makes the convolution algebra of the Lie groupoid $\Gamma \rightrightarrows \mathcal{A}$ independent of the choice of left Haar system, since they are intrinsic objects to the Lie groupoid. This also makes the representation of the Lie groupoid independent of the choice of smooth left Haar system. As stated above, the choice of left Haar system is made relative by the equivalence established on the space of densities $C^\infty(\Gamma,\Omega(\Gamma))$ by the action of the smooth net $\mathscr{K} \simeq [0,1)$ which we simply put as $(\Gamma^\Gamma)^I$. We give this as a proposition as follows.
\begin{prop}
The convolution algebra $C^\infty(\Gamma,\Omega(\Gamma))$ has a smoothing action of $\mathscr{K} \simeq [0,1)$.
\end{prop}
\begin{proof}
Because $\Omega(\Gamma)_\gamma$ are trivial line bundles on $\Gamma$, we have $\Omega(\Gamma) \cong \Gamma \times \R = \R\Gamma$ (or $\Gamma \times \mathbb{C} = \mathbb{C}\Gamma$). Therefore, $C^\infty_c(\Gamma)$ can be identified with the space of \emph{smooth sections} $C^\infty_c(\Gamma,\Omega(\Gamma))$ which are the smooth functions $\Gamma \to \Omega(\Gamma) \simeq \R\Gamma$; i.e. the bisections with the action of $[0,1)$. (cf. \cite{Husemoller94}). This gives a sectional dynamics on the $t$-fibres.
\end{proof}
The equivalence defined by this action makes the choice of left Haar system of measures defined by the smooth functions on the arrows not to be unique. This was what Paterson \cite{Paterson99} meant by positing that Alain Connes' approach to the convolution algebra in \cite{Connes94} makes the definition of the convolution algebra, and the representations (the unitary and irreducible) of a Lie groupoid $\mathcal{G}$ independent of the choice of smooth left Haar system. The above proposition says that the infinitesimal approach is equivalent to smoothing net $\mathscr{K}$-action on the space of local bisections isomorphic to the arrows in $\Gamma(x,-)$.

\subsection{Infinite dimensionality of the Representation Space}
Let $\nu$  be a probability measure on $\mathcal{A}$. Then a suitable Hilbert space $\mathcal{H}$ for the representation of the Lie groupoid $\Gamma \rightrightarrows \mathcal{A}$ arising from the locally convex partial $^*$-algebra $(\mathcal{A},\Gamma, \cdot,*,\tau)$ is the Hilbert bundle defined as a triple $(\mathcal{A},\mathcal{H},\nu)$, where $\mathcal{A}$ is the locally convex Hausdorff space, and $\nu$ is an invariant or quasi-invariant probability measure on $\mathcal{A}$, and $\mathcal{H}$ is the collection of Hilbert spaces $\{H_x = L^2(\Gamma(x,-),\mu^x)\}$ indexed by the elements of $\mathcal{A}$, which is a Hilbert bundle over $\mathcal{A}$.

The definitions of sections of the Hilbert bundle $\mathcal{H}$ and their nets which determine the inner product norm on the bundle are done in accordance with these formulations.  The identification of the arrows with the local bisections modifies the definition of the unitary operators $\ell(\varphi) : H_{s(\gamma)} \to H_{t(\gamma)}$. Subsequently, a local bisection $\varphi \in \mathcal{B}_\ell(\mathcal{G})$ can also be used to define the $C^*$-representation given by the (densities) map $\gamma \mapsto f(\gamma)\ell(\gamma)$ for each $x \in \mathcal{A}, f \in C_c(\Gamma)$. In this case, the map $\varphi \mapsto f(\gamma)\varphi$ is a section of the Hilbert bundle (cf. \cite{Paterson99}). Thus, given a trivialization $U$ at $x \in \mathcal{A}$, the Hilbert space $H_x = L^2(\Gamma(x,-),\mu^x)$ can also be given in terms of the local bisection $H_x = L^2(\mathcal{B}_\ell(U),\mu^x)$.

The image of the sections are functions defined on the arrows terminating at $x \in \mathcal{A}$, which are square integrable. The net of smooth sections follows on the net of local bisections which has a smooth structure (of a Lie pseudogroup) as compared to the sequence of sections defined in \cite{Paterson99}. The inner product $\langle \varphi_1(x),\varphi_2(x) \rangle$ is the diffeomorphism-invariant product of two local bisections defined as $$\displaystyle \int (\varphi_1 \star \varphi_2)(x)d\nu(x) = \int \varphi_1(t \circ \varphi_2(x))\varphi_2(x)d\nu(x).$$ This takes care of the convolution of two densities and the diffeomorphism invariance of integration on a smooth manifold. Following \cite{Paterson99}, a fundamental net is therefore defined as follows.
\begin{defx}
A net $(\varphi_n)$ of sections is said to be \emph{fundamental} if for each pair of indices $m,n$ the function $\displaystyle x \mapsto \langle \varphi_m(x),\varphi_n(x) \rangle = \int \varphi_m(t\circ \varphi_n(x))\varphi_n(x)d\nu(x)$ is $\nu$-measurable on $\mathcal{A}$; and for each $x \in \mathcal{A}$, the images $\varphi_n(x)$ of the net span a dense subspace of $H_x = L^2(\Gamma(x,-),\mu^x)$.
\end{defx}
\begin{prop}
A net $(\varphi_n)$ of local bisection $B_\ell(U)$ of $\Gamma \rightrightarrows \mathcal{A}$ is fundamental.
\end{prop}
\begin{proof}
First, bisections $\varphi$ for the Lie groupoids are sections satisfying the definition above. Second, since bisections are used to defined left translations $L_\varphi$ on a Lie groupoid (see \cite{Mackenzie2005}), the image of the net $(\varphi_n)$ span a dense subspace of the Hilbert space $H_x = L^2(\Gamma(x,-),\mu^x)$ by the openness of the target map $t : \Gamma(x,-) \to \mathcal{A}$ which forms a net of (local) diffeomorphisms $t \circ \varphi_n$. Finally, the transitive action of $\Gamma$ on $\Gamma(x,-)$ also points to the denseness of the span of a net of bisections, for a restriction of the left translation by a bisection is open as stated in $(1.4)$ of \cite{Mackenzie2005}.
\end{proof}
\begin{rem}(see \cite{Paterson99})
The above result points to the relation between elements of the base space $x \in \mathcal{A}$ and the fundamental nets $(\varphi_n)$, such that:\\
(1) The smooth (fundamental) net $(\varphi_n)$ can be considered the orthonormal basis of the bundle since the Gram-Schmidt process can be used to convert the image of the net $\{\varphi_n(x)\}$ to an orthonormal basis for $H_x = L^2(\Gamma(x,-),\mu^x)$ for each $x \in \mathcal{A}$. \\
(2) A section $\varphi$ is measurable when the action of a fundamental net $(\varphi_n)$ on it by inner product is measurable; that is,  each function $x \mapsto \langle \varphi(x), \varphi_n(x) \rangle$ of the net is measurable. It follows that the smooth fundamental net $(\varphi_n)$ defines the notion of measurability for sections. This extends the connection established between local bisections and the sections of the Hilbert bundle $\mathcal{H}$ to the sections of line bundle defining densities.\\
(3) Thus, the Hilbert bundle $\mathcal{H} = L^2(\mathcal{A},\{H_x\},\nu)$ is the space of measurable sections $\varphi$ with a relation $\sim$ defined by the convergence of nets $\varphi_n$, which defines a $\nu$-integrable function $x \mapsto ||\varphi(x)||^2_2$, with inner product $\displaystyle \langle \varphi,\psi \rangle = \int_{\mathcal{A}}\langle \varphi(x),\psi(x) \rangle d\nu(x)$. The related norm is $\displaystyle ||\varphi||^2 = \int_{\mathcal{A}} \langle \varphi(x),\varphi(x) \rangle d\nu(x)$, while the $L^2$-norm is $\displaystyle ||\varphi||^2_2 = \int_{\mathcal{A}}||\varphi(x)||^2d\nu(x)$. The net convergence constitutes the measurable sections as generalized quantities, as defined in \cite{Marsden68}. This gives the following corollary.
\end{rem}
\begin{coro}
The relation $\sim$ defined by the net $\varphi_n$ constitutes a class $[\varphi]$ of sections with an action of the smooth net $\mathscr{K}$.
\end{coro}
\begin{proof}
This follows from the convolution action of $\mathscr{K}$ on the local bisection $\varphi$ which defines the net of bisections $(\varphi_n)$. Thus, the class $[\varphi]$ of a section $\varphi$ has the inner product action $\langle \varphi(x),\varphi_n(x) \rangle$ of a fundamental net $(\varphi_n)$ in terms of measurability. Thus, the fundamental net $(\varphi_n)$ spans a dense subspace of each $L^2(\Gamma(x,-),\mu^x)$ on any $x \in \mathcal{A}$, and defines the classes of sections of the Hilbert bundle $\mathcal{H} = L^2(\mathcal{A},\{H_x\},\nu)$, with each class a net $\varphi_n$ of sections converging to a measurable section $\psi$.
\end{proof}
As was noted earlier, the fibres are not all of same dimension for the Lie groupoid $\mathcal{G} = \Gamma \rightrightarrows \mathcal{A}$; the right multipliers $\Gamma(x,-)$ and left multipliers $\Gamma(-,x)$ are not the same for every $x \in \mathcal{A}$. We are interested in the infinite dimensional fibres which are inductive limit of these fibres. Thus, we have a series of containment for the Hilbert space $H_x = L^2(\Gamma(x,-),\mu^x) : x \in \mathcal{A}$, such that $H_{x_1} \supset H_{x_2} \supset \cdots \supset L^2(\Gamma(x,x),\mu^x)$, where the right multipliers are also ordered $\Gamma(x_1,-) \supset \Gamma(x_2,-) \supset \cdots $ (and respectively the left multipliers $\Gamma(-,x_1) \supset \Gamma(-,x_2) \supset \cdots $). We need the following proposition to give the result on the inductive system of locally convex Lie groupoids.
\begin{prop}\cite{Mackenzie2005}
Let $\Gamma \rightrightarrows \mathcal{A}$ be a Lie groupoid, and $\sigma : U \to V$ be a diffeomorphism from $U \subset \Gamma$ open to $V \subseteq \Gamma$ open, and let $f : B \to C$ be a diffeomorphism from $t(U) = B \subseteq \mathcal{A}$ to $t(V) = C \subseteq \mathcal{A}$, such that $s\circ \sigma = s, t\circ \sigma = f\circ \sigma$ and $\sigma(\gamma \eta) = \sigma(\gamma)\eta$ whenever $(\gamma,\eta) \in \Gamma * \Gamma, \gamma \in U$ and $\gamma \eta \in U$. Then $\sigma$ is the restriction to $U$ of a unique local translation $L_\varphi : \Gamma(B,-) \to \Gamma(C,-)$ where $\varphi \in \mathcal{B}_\ell(U)$.
\end{prop}
\begin{prop}
The orbits of $\Gamma$ by the left translation of a bisection $L_\varphi$ defines inductive system of Lie groupoids, which reflects as a connected path on $\mathcal{A}$.
\end{prop}
\begin{proof}
Since the partial algebra structure is encoded by the local bisections $\mathcal{B}_\ell(\Gamma)$, given that the target map $t$ restricted to a source fibre $t_x : \Gamma(-,x) \to \mathcal{A}$ is of constant rank, it gives a diffeomorphism $t \circ \varphi : \mathcal{A} \to \mathcal{A}$. This means that the left translation $L_\varphi$ is transitive on the fibre. Hence, there exists $C \subset \mathcal{A}$ such that $\Gamma(C,x)$ is closed under left multiplication or translation.

By definition, $\Gamma$ defines a relation on $\mathcal{A}$. A conjugate class $[x]$ of $x \in \mathcal{A}$ by this relation is the set $\{y \in \mathcal{A} : x\cdot y \in [x] \subseteq \mathcal{A}, (x,y) \in \Gamma \}$; and $x \cdot x \in [x]$. This defines an inductive system of locally convex partial $^*$-algebras as follows. Let $I$ be a directed set, with order denoted by $\gg$. Given the Lie groupoid $\Gamma \rightrightarrows \mathcal{A}$, we define a net of subsets of $\Gamma$ (arrows) $\{\Gamma_\alpha\}_{\alpha \in I}$ of the Lie groupoid such that the following are satisfied. \\
(1) $\Gamma_\alpha \subset \Gamma$, for each $\alpha$, and $\Gamma_\alpha \subseteq \Gamma_\beta$ for $\beta \gg \alpha$ and $\mathcal{A}_\alpha \neq \mathcal{A}_\beta, \alpha \neq \beta$;\\ (2) If $\alpha, \beta \in I$ and $\beta \gg \alpha$, then $\Gamma_\alpha \rightrightarrows \mathcal{A}_\alpha$ is a subgroupoid or a restriction of the Lie groupoid for $\mathcal{A}_\alpha \subset \mathcal{A}$.\\
(3) The natural homomorphic embeddings of $\mathcal{A}_\alpha \to \mathcal{A}$ and $\mathcal{A}_\alpha \to \mathcal{A}_\beta$ is by the extension of a left translation by $\varphi$; that is, $L_\varphi(\gamma) = \sigma(\gamma)$, where $\sigma$ is a restriction of the left translation $L_\varphi$, $\varphi \in \mathcal{B}_\ell(\Gamma)$. Similarly, we can write the left translation $L_\varphi$ as $\sigma_\alpha$ and $\sigma_\beta$ which are the embedding $\Gamma_\alpha \to \Gamma$ and $\Gamma_\beta \to \Gamma$ respectively. Then $\sigma_{\beta\alpha} : \Gamma_\alpha \to \Gamma_\beta$ which satisfies the cocycle condition $\sigma_{\gamma\alpha} = \sigma_{\gamma\beta} \circ \sigma_{\beta\alpha}$ and $\sigma_\beta \circ \sigma_{\beta\alpha} = \sigma_\alpha$ for $\gamma \gg \beta \gg \alpha$.\\ (4) The linear span of $\underset{\alpha \in I}\bigcup \sigma_\alpha(\mathcal{A}_\alpha)$ is $\mathcal{A}$. \\ (5) That the system of Lie subgroupoids $\{(\Gamma_\alpha \rightrightarrows \mathcal{A}_\alpha, \sigma_\alpha), (\sigma_{\beta\alpha})_{\alpha, \beta \in I} : \beta \gg \alpha\}$ is compatible with the algebraic and topological structures of the locally convex partial $^*$-algebra follows from the compatibility of these structures within a topological (Lie) groupoid. \\ Thus, the system is an inductive system of Lie groupoids with the Lie groupoid $\Gamma \rightrightarrows \mathcal{A}$ as its inductive limit.
\end{proof}
\begin{coro}
The inductive system $\{\Gamma_\alpha \rightrightarrows \mathcal{A}_\alpha, \sigma_\alpha, \sigma_{\beta\alpha}, \alpha, \beta \in I : \beta \gg \alpha\}$ is equivalent to the inductive system $\{(\mathcal{A}_\alpha,\varphi_\alpha,\tau_\alpha)_{\alpha \in I},(\varphi_{\beta\alpha})_{\alpha,\beta \in I}:\beta \gg \alpha, \tau\}$ of locally convex partial $^*$-algebras defined in \cite{Ekhaguere2007}, whose inductive limit is the locally convex $^*$-algebra $(\mathcal{A},\Gamma, \cdot,*,\tau)$.
\end{coro}
\begin{proof}
The proof of this is an immediate consequence of the construction or formulation of the Lie groupoid $\Gamma \rightrightarrows \mathcal{A}$. For by the formulation, as we have seen above, the Lie groupoid $\Gamma_\alpha \rightrightarrows \mathcal{A}_\alpha$ is always a subgroupoid of the Lie groupoid of pairs $\mathcal{G}$ where $\Gamma = \mathcal{A} \times \mathcal{A}$. Thus, the Lie groupoid of pairs $\mathcal{G} \rightrightarrows \mathcal{A}$ on the locally convex topological space $\mathcal{A}$ models a locally convex $^*$-algebra $(\mathcal{A},\Gamma, \cdot, *, \tau)$.
\end{proof}
\section{The Representations of the Lie Groupoid}
Given the smooth system of Haar measures $\{\mu^x\}_{x \in \mathcal{A}}$ supported on the $t$-fibres $\Gamma(x,-)$ as described above, each Haar measure is associated to a Hilbert space $H_x = L^2(\Gamma(x,-), \mu^x)$ such that each arrow $\gamma \in \Gamma(t(\gamma),-)$ defines a unitary operator $\ell(\gamma) : H_{s(\gamma)} \to H_{t(\gamma)}$. This gives rise to a unitary representation of the Lie groupoid $\Gamma \rightrightarrows \mathcal{A}$ on the space $\mathcal{U}(\mathcal{H})$ of unitary operators on the Hilbert bundle $\mathcal{H} = \{H_x\}_{x \in \mathcal{A}}$. The unitary representation is then used in the definition of the $C^*$-representation of the groupoid convolution algebra $C(\Gamma)$, which is a representation defined by the (densities) map $\gamma \mapsto f(\gamma)\ell(\gamma)$ over $\Gamma(x,-)$ for each $x \in \mathcal{A}, f \in C_c(\Gamma)$ with respect to the Haar measure $\mu^x$.

The unitary representation $\ell$ of the Lie groupoid $\Gamma \rightrightarrows \mathcal{A}$ and the $C^*$-representation of its convolution algebra on the Hilbert space bundle $\mathcal{H} = L^2(\mathcal{A},\{H_x\},\nu)$ follow directly on the above formulations. As we have noted already, given the probability measure $\nu$ which is quasi-invariant on $\mathcal{A}$, it defines the quasi-invariant measures $m, m^2$ and their inverses $m^{-1},(m^2)^{-1}$ on $\Gamma, \Gamma^{(2)}$ respectively, and satisfy the requirements for measurability of inversion and product. The definitions of these associated measures $m, m^{-1}, m^2, m_o$ to $\nu$ and $\{\mu^x\}_{x \in \mathcal{A}}$ are given in \cite{Paterson99}. So, following Paterson, we define the representation of the Lie groupoid $\Gamma \rightrightarrows \mathcal{A}$ as follows.
\begin{defx}
A representation of the locally convex groupoid $\Gamma \rightrightarrows \mathcal{A}$ is defined by a Hilbert bundle $(\mathcal{A},\{H_x\},\nu)$ where $\nu$ is a quasi-invariant measure on $\mathcal{A}$ to which the measures $m, m^{-1}, m^2, m_o$ are associated; and for each $\gamma \in \Gamma$, there is a unitary element $\ell(\gamma) : H_{s(\gamma)} \to H_{t(\gamma)}$ such that \\ (i) $\ell(x)$ is the identity map on $H_x$ for all $x \in \mathcal{A}$; \\ (ii) $\ell(\gamma_1\gamma_2) = \ell(\gamma_1)\ell(\gamma_2)$ for $m^2$-a.e. $(\gamma_1,\gamma_2) \in \Gamma^2$; \\ (iii) $\ell(\gamma)^{-1} = \ell(\gamma^{-1})$ for $m$-a.e. $\gamma \in \Gamma$; \\ (iv) for any $\xi, \eta \in L^2(\mathcal{A},\{H_x\},\nu)$, the function \[\gamma \mapsto \langle \ell(\gamma)\xi(s(\gamma)), \eta(t(\gamma))\rangle \eqno{(9)} \] is $m$-measurable on $\Gamma$.
\end{defx}
The inner product is defined since $\xi(s(\gamma)) \in H_{s(\gamma)}$ and translated by $\ell(\gamma)$ to $\ell(\gamma)\xi(s(\gamma)) \in H_{t(\gamma)}$, and $\eta(t(\gamma)) \in H_{t(\gamma)}$. So the inner product is on the fibre $H_{t(\gamma)}$. The representation is given as a triple $(\nu,\mathcal{H},\ell)$. The identification between an arrow and a local bisection $\gamma \leftrightarrow \varphi$ means that we can use either for the representation. These give rise to the \emph{trivial} and the \emph{left regular representations} of the Lie groupoid given as follows. (cf. \cite{Paterson99}, p.93).

\subsection{Trivial and Left Regular Representations}
Working on the field of real numbers, we leave aside the complex trivial representation and focus on the real. We replace the complex numbers $\mathbb{C}$ with the ordered space or connected component $[0,1]$ in the trivial representation $\ell_t : \mathbb{C}_{s(\gamma)} \to \mathbb{C}_{t(\gamma)}$, in which $\ell_t$ is an identity map on $\mathbb{C}$. In this case, the natural or trivial representation  $\ell_t$ is on the trivial bundle $\mathcal{A} \times [0,1]$ over $\mathcal{A}$. Thus, each fibre $H_x$ is again 1-dimensional Hilbert space $[0,1]_{x \in \mathcal{A}} \simeq \mathcal{A} \times \mathscr{K}$. Hence, $\ell_t(\gamma) = 1 \in [0,1]$ for all $\gamma \in \Gamma$; making $\ell_t : [0,1]_{s(\gamma)} \to [0,1]_{t(\gamma)}$ an identity map on $[0,1]$.

The second natural representation is the left regular $\ell_r$ representation which is built on the trivial. It is defined on the fibres $H_x = L^2(\Gamma(x,-))$ (see \cite{Paterson99}, p.107; \cite{Renault80}, p.55). The convolution algebra $C_c(\Gamma)$ is made a space of continuous (smooth) sections of $\{H_x\}$ by identification of each $\varphi \in C_c(\Gamma)$ with the section $x \to \varphi|_{\Gamma(x,-)} \in C_c(\Gamma(x,-) \subset L^2(\Gamma(x,-))$. Hence, according to definition, any pair of sections $\varphi,\psi \in C_c(\Gamma)$ is required to define a $\nu$-measurable map by inner product $\gamma \mapsto \langle \varphi(\gamma), \psi(\gamma) \rangle$. This is satisfied since $\varphi\overline{\psi} \in B(\mathcal{H})$-bounded operators-and the restriction $(\varphi\overline{\psi})|_\mathcal{A} = (\varphi\overline{\psi})^o \in B_c(\mathcal{A})$.

The generation of the Hilbert bundle from these sections is shown as follows. From our construction and definition of the Lie groupoid $\Gamma \rightrightarrows \mathcal{A}$, and following also from the definition of locally compact groupoid in (\cite{Paterson99}, Definition 2.2.1), there is a countable family $\mathcal{C}$ of compact (convex) Hausdorff subsets of $\Gamma$ such that the family $\{C^o, C \in \mathcal{C}\}$ of interiors of $\mathcal{C}$ is a basis for the topology of $\Gamma$ (we have used or defined this open basis $\mathscr{U}$ of locally convex topology of $\Gamma$ above as $N_x \subset \Gamma(x,-)$.) For $C \in \mathcal{C}$, there exists a sup-norm dense countable subset $A_C$ (of sections) of $C_c(C^o)$ corresponding to $C \in \mathcal{C}$; a fundamental net of sections is given by the sums of the functions in $A_C \subset C_c(C^o)$ as $C$ ranges over the family $\mathcal{C}$ (an ultrafilter), which is same as the set of images of all the local bisections $B_\ell(\mathcal{G})$.

Subsequently, $\{H_x\}$ is a Hilbert bundle, and and the map $\ell_r(\gamma) : H_{s(\gamma)} \to H_{t(\gamma)}$ defined by $(\ell_r(\gamma))(f)(\gamma_1) = f(\gamma^{-1}\gamma_1)$, with $f \in L^2(\Gamma(s(\gamma),-))$ and $\gamma_1 \in \Gamma(t(\gamma),-)$, is a unitary representation given as follows (see \cite{Paterson99}). First, $\ell_r(\gamma)$ is an extension of a bijective isometry \[L^1(\Gamma(s(\gamma),-),\mu^{s(\gamma)}) \to L^1(\Gamma(t(\gamma),-), \mu^{t(\gamma)}), f \mapsto \gamma * f; \eqno{(10)}\] in the sense that $L^2(\Gamma(s(\gamma),-),\mu^{s(\gamma)}) \subset L^1(\Gamma(s(\gamma),-),\mu^{s(\gamma)})$; and it defines a translative (transitive) action of $\Gamma$ on $t$-fibres.

Second, the restriction of the $\ell_r$ to $L^2(\Gamma(s(\gamma),-),\mu^{s(\gamma)})$ is a representation of $\gamma \in \Gamma$ as a unitary operator $\gamma * f$ given as \[ H_{s(\gamma)} \to H_{t(\gamma)}, (\gamma * f)(\gamma_1) = f(\gamma^{-1}\gamma_1) = f(\gamma^{-1})f(\gamma_1). \eqno{(11)}\]  The restriction holds for $1 < p \leq \infty$ (see \cite{Paterson99}, p.34). So, $\ell_r$ is a (unitary) representation of $\Gamma \rightrightarrows \mathcal{A}$ on the Hilbert bundle $(\mathcal{A},\mathcal{H},\nu)$ for it satisfies the other conditions of definition.
\begin{rem}
The translative action of $\Gamma$ on the $t$-fibres which are open subspaces of the locally convex Lie groupoid $\Gamma$ can also be given by the action of the local bisections of the Lie groupoid, which are always diffeomorphic to the open subspaces of a Lie groupoid; that is, $L \longleftrightarrow \varphi$, where $L$ is an open subspace of $\Gamma$ and $\varphi$ is a local bisection of $\Gamma$.
\end{rem}

\subsection{The $C^*$-Representation of the Convolution Algebra}
The unitary representation of the Lie groupoid $\Gamma \rightrightarrows \mathcal{A}$ on $\mathcal{U}(\mathcal{H})$-space of unitary operators on the Hilbert bundle $\mathcal{H}$ comprising of the Hilbert spaces $(L^2(\Gamma(t(\gamma),-), \mu^{t(\gamma)})$-is connected to $^*$-representation of the convolution algebra $C_c(\Gamma)$ on the space of operators on the same bundle space $\mathcal{H}$. The definition of the $^*$-representation presupposes the definition of involution or $I$-norm, the role of which is to keep the involution isometric on the convolution algebra $C_c(\Gamma)$. (See also \cite{Renault80}, p.51).
\begin{defx}\cite{Paterson99}
Given the locally convex (compact) groupoid $\Gamma \rightrightarrows \mathcal{A}$, with $C_c(\Gamma)$ as the space of continuous functions on $\Gamma$ supported on compact sets, the following norms are defined on $C_c(\Gamma)$ as follows.
\[ ||f||_{I,t} = \underset{x \in \mathcal{A}}\sup \int_{\Gamma(x,-)}|f(\gamma)|d\mu^x(\gamma); \;   ||f||_{I,s} = \underset{x \in \mathcal{A}}\sup \int_{\Gamma(-,x)}|f(\gamma)|d\mu_x(\gamma) \eqno{(11)}\] \[ ||f||_I = \max \{||f||_{I,t}, ||f||_{I,s} \} \; \text{is the I-norm}. \]
\end{defx}
 We now modify Paterson's formulations to suit our locally convex Lie groupoid $\Gamma \rightrightarrows \mathcal{A}$ as follows.
\begin{prop}
Let $C$ be a compact (convex) subset of $\Gamma$ and $\{f_\alpha\}$ be a net in $C_c(\Gamma)$ such that every $f_\alpha$ vanishes outside $C$; that is, $\alpha < dist((x,x), \partial C)$. Suppose that $f_\alpha \to f$ uniformly in $C_c(\Gamma)$. Then $f_\alpha \to f$ in the $I$-norm of $C_c(\Gamma)$.
\end{prop}
\begin{proof}
Using the countable family $\mathcal{C}$ of compact (convex) Hausdorff subsets of $\Gamma$ as given above, the compactness of $\mathcal{C}$ implies an open covering $U_1, \cdots, U_n$ of $C$ by Hausdorff subsets of open Hausdorff sets $V_1, \cdots, V_n$ (take these to be the images of local bisections $B_\ell(\Gamma)$) such that the closure of each $U_i$ in $V_i$ is compact. Let $F_i \in C_c(V_i)$ be such that $F_i \geq \chi_{U_i}$. In particular, $F_i$ is positive. Let $F = \overset{n}{\underset{i=1}\sum}F_i$. (If we use the net of bisections, then $F = \underset{\alpha}\sum F_\alpha$.) Then $F \in C_c(\Gamma)$ and $F \geq \chi_C$. Hence $|f_\alpha - f| \leq |f_\alpha - f|F$, and we have \[||f_\alpha - f||_{I,t} = \underset{x \in \mathcal{A}}\sup \int_{\Gamma(x,-)}|f_\alpha(t) - f(t)|d\mu^x(\gamma) \] \[ \leq \underset{x \in \mathcal{A}}\sup \int_{\Gamma(x,-)}|f_\alpha(t) - f(t)|F(\gamma)d\mu^x(\gamma) \] \[ = ||f_\alpha - f||_\infty||F^o||_\infty \to 0, \; \text{as } \alpha \to 0. \]
Similarly, $||f_\alpha - f||_{I,s} \to 0$, and so the same for the $I$-norm.
\end{proof}
So $C_c(\Gamma)$ is a normed $^*$-algebra under $I$-norm, with a $I$-norm continuous representation on a Hilbert space. The separable normed $^*$-algebra $C_c(\Gamma)$ generates separable $C^*$-algebras on the separable Hilbert bundle $\mathcal{H}$. The bundle space $\mathcal{H}$ is separable because the fibres are separable Hilbert spaces. With this it follows that $\pi_\ell : C_c(\Gamma) \to B(\mathcal{H})$ defined as \[ \langle \pi_\ell(f)\xi, \eta \rangle = \int_{\Gamma}f(\gamma)\langle \ell(\gamma)(\varphi(s(\gamma))),\psi(t(\gamma))\rangle dm_o(\gamma), \eqno{(12)} \] is a $^*$-representation of the convolution algebra of the Lie groupoid. This follows from Paterson's statement and proof of the results of Renault on $^*$-representation of Lie groupoids in \cite{Paterson99}.
\begin{prop}(cf. \cite{Paterson99},Proposition 3.1.1)
The equation \[\langle \pi_\ell(f)\varphi, \psi \rangle = \int_{\Gamma}f(\gamma)\langle \ell(\gamma)(\varphi(s(\gamma))), \psi(t(\gamma))\rangle dm_o(\gamma) \] defines a representation $\pi_\ell$ of $C_c(\Gamma)$ of norm $\leq 1$ on the bundle $\mathcal{H} = L^2(\mathcal{A},\{H_x\},\nu)$.
\end{prop}
\begin{rem}
From the result we deduce a relationship between the two representations; the unitary representation of the Lie groupoid $\Gamma \rightrightarrows \mathcal{A}$ on the Hilbert bundle $\mathcal{H} = L^2(\mathcal{A},\{H_x\},\nu)$ and the $C^*$-representation of its convolution algebra $C^\infty(\Gamma,\Omega(\Gamma)) \simeq C^\infty(\mathcal{A})|_\Gamma$ on a dense subspace of the $C^*$-algebra $B(\mathcal{H})$ of bounded operators on the Hilbert bundle $\mathcal{H}$. This relationship, according to Paterson, is basic to the fundamental theorem of analysis on locally compact groupoids. It has to do with the fact that every representation of the convolution algebra $C_c(\mathcal{G})$ of a locally compact groupoid $\mathcal{G}$ is some $^*$-representation $\pi_\ell$ of the unitary representation $\ell$ of the groupoid.

Thus, the $C^*$-representation of the convolution algebra $C^\infty(\Gamma,\Omega(\Gamma))$ on a dense subspace of the $C^*$-algebra $B(\mathcal{H})$ of bounded operators on the Hilbert bundle $\mathcal{H}$ is the $^*$-representation $\pi_\ell$ of the unitary representation $\ell$ of the locally convex groupoid $\Gamma \rightrightarrows \mathcal{A}$. This is stated in the the following theorem.
\end{rem}
\begin{thm}(Cf. \cite{Paterson99})
Given the locally convex Lie groupoid $\Gamma \rightrightarrows \mathcal{A}$. The representation of $C_c(\Gamma,\Omega(\Gamma))$ is the $^*$-representation $\pi_\ell$ of the unitary representation $\ell :\Gamma \to \mathcal{U}(\mathcal{H})$; and the correspondence $\ell \mapsto \pi_\ell$ preserves the natural equivalence between the two representations.
\end{thm}
\begin{proof}
This natural equivalence rests on the idea of groupoid equivalence which captures the partial symmetry encoded by the locally convex Lie groupoid $\Gamma \rightrightarrows \mathcal{A}$. This partial symmetry is portrayed or captured in the $(\Xi,\Gamma)$-equivalence of the $t$-fibres $\Gamma(x,-)$. The isomorphism of the two algebras $C_c(\Gamma,\Omega(\Gamma)) \simeq C^\infty(\mathcal{A})|_\Gamma$ brings this partial symmetry to the fore. For it shows that the partial symmetry of the convolution algebra $C_c(\Gamma,\Omega(\Gamma))$ defined on the arrows of the Lie groupoid is isomorphic or same as the partial symmetry of the smooth algebra defined on $\mathcal{A}$ but restricted to the relation $\Gamma$ on $\mathcal{A}$.

Given that the $t$-fibre $\Gamma(x,-)$ (or the $s$-fibre $\Gamma(-,x)$) is the orbit of the relation $\Gamma$ through $x$, it constitutes a representation space of both the smooth algebra and the convolution algebra of the Lie groupoid $\Gamma \rightrightarrows \mathcal{A}$. The $(\Xi,\Gamma)$-equivalence implies that the two actions, $\Xi$-action and $\Gamma$-action commute; while the former defines the unitary representation, the latter defines the $^*$-representation. This gives rise to the natural equivalence $\ell \leftrightarrow \pi_\ell$ of the two representations.
\end{proof}
The following lemma on the groupoid equivalence clarifies this natural equivalence and completes the proof of the theorem. It also extends the above result on the structure of the Lie groupoid we have modelled on the locally convex partial $^*$-algebra.
\begin{lem}
Given that $\Xi = \underset{x \in \mathcal{A}}\bigsqcup \Gamma(x,x)$ is a Lie groupoid. The $(\Xi, \Gamma)$-equivalence of the $t$-fibres $\Gamma(x,-)$ gives rise to a $C^*$-isomorphism.
\end{lem}
\begin{proof}
We have seen above that the sectional transitivity of the Lie groupoid $\Gamma \rightrightarrows \mathcal{A}$ is well reflected in the definition of its convolution algebra $C_c(\Gamma) = C^\infty(\mathcal{A})_\Gamma$, since a convolution follows the product operation between a pair of arrows $\gamma, \eta \in \Gamma$. Thus, an open submanifold $\Gamma(x,-)$ or $\Gamma(-,x)$ reflects a (path) connected maximal proper subset $B$ of the locally convex topological vector space $\mathcal{A}$. The maximality of $B \subset \mathcal{A}$ with respect to path connectedness makes it either open and dense in $\mathcal{A}$ or a closed proper subset of $\mathcal{A}$. This is clear from the inductive system defined above (see also \cite{Ekhaguere2007}).

Subsequently, from the definition of a $t$-fibre $\Gamma(x,-)$ as a $(\Xi, \Gamma)$-equivalence in \cite{MuhRenWil87}, the existence of a special equivalence between the $C^*$-algebras $C^*(\Xi)$ and $C^*(\Gamma)$ was established, which leads to isomorphism between $C^*(\Gamma)$ and $C^*(\Xi)\otimes \mathcal{K}(L^2(\mathcal{A},\nu))$, where $\mathcal{K}$ is the set of compact operators. On the other hand, from the definition of the Hilbert bundle $\mathcal{H} =  \underset{x \in \mathcal{A}}\bigsqcup H_x$, where $H_x = L^2(\Gamma(x,-),\mu^x)$, it is evident that $\mathcal{H}$ is a left $\Xi$-principal bundle.

Finally, from the position of $\Gamma \rightrightarrows \mathcal{A}$ as the limit of the inductive system $\{(\Gamma_\alpha \rightrightarrows \mathcal{A}_\alpha, \sigma_\alpha), (\sigma_{\beta\alpha})_{\alpha,\beta \in I}: \beta \gg \alpha\}$, which implies that $\Gamma$ is transitive on a dense subset of $\mathcal{A}$, we conclude that the $B(\mathcal{H})$ which has the representation $\ell \to \pi_\ell$ is isomorphic to $C^*(\Xi)\otimes \mathcal{K}(L^2(\mathcal{A},\nu))$. That is, $B(\mathcal{H}) \simeq C^*(\Xi)\otimes \mathcal{K}(L^2(\mathcal{A},\nu))$.
\end{proof}
\begin{rem}
This isomorphism depends on the probability measure $\nu$ on $\mathcal{A}$ which is unique to each unit $x \in \mathcal{A}$, and on the density or local bisection $\varphi \in \mathcal{B}_\ell(B_x)$ defined on $x \in \mathcal{A}$, and related to the system of Haar measures. The implication is that the inductive system also extends to the $C^*$-algebras, whereby the unbounded operators are the inductive limit of bounded operators.
\end{rem}
\section{conclusion}
The groupoid characterizations of the partial algebras characterized in \cite{Ekhaguere2007} have helped us to arrive at a clearer understanding of the structures of these partial algebras. Most important is the Lie groupoid characterization of the locally convex partial $^*$-algebras $(\mathcal{A},\Gamma,\cdot, ^*, \tau)$ which clearly demonstrates the facility of (Lie) groupoid framework to handle pathological spaces. This facility is aptly captured in the correspondence between groupoid equivalence and isomorphism of groupoid $C^*$-algebras established in \cite{MuhRenWil87}, which is exemplified in this work in a special way.

The $(\Xi,\Gamma)$-equivalence of the $t$-fibres $\{\Gamma(x,-) : x \in \mathcal{A}\}$ relates to the Lie groupoid  of pairs $\mathcal{G} \rightrightarrows \mathcal{A}$ as the inductive limit of the inductive system of Lie groupoids $\Gamma_\alpha \rightrightarrows \mathcal{A}_\alpha$ given above. Hence, the results show in the case of the locally convex Lie groupoid $\Gamma \rightrightarrows \mathcal{A}$ we have formulated from the locally convex partial $^*$-algebra $(\mathcal{A},\Gamma,\cdot, ^*, \tau)$, that both the equivalence of (Lie) groupoids and that of their $C^*$-algebras follow from the equivariant actions of the smoothing algebra $\mathscr{K}$. Hence, the fact that they are equivariant $\mathscr{K}$-spaces contributes to the correspondence between the two representations and the isomorphism of $C^*$-algebras.

In addition, the smooth equivalence also presents the Lie group bundle $\Xi = \{\Gamma(x,x) : x \in \mathcal{A}\}$ as the deductive limit of a deductive system $\{\Gamma(y,x), y_\varepsilon \to x : \varepsilon \to 0\}$, which implies the convergence of every closed manifold $\Gamma(y,x) \to \Gamma(x,x)$. This could be considered a deformation of the Lie groupoid $\Gamma(\mathcal{A},x)$ to the Lie group $\Gamma(x,x)$ at each unit $x \in \mathcal{A}$ using $\varepsilon \in [0,1)$ as a deformation parameter; and the deformation of the (transitive) Lie groupoid $\Gamma \rightrightarrows \mathcal{A}$ to the Lie group bundle $\Xi$ which is an (intransitive) Lie groupoid. (Cf. \cite{Mackenzie2005}, 1.5.9). The notion of deformation as connected to (Lie) groupoid is treated in \cite{Mackenzie2005}, 1.6.21.

\bibliographystyle{amsplain}

\providecommand{\bysame}{\leavevmode\hbox to3em{\hrulefill}\thinspace}

\end{document}